\documentclass [smallextended] {article}

\usepackage{amsmath}
\usepackage{amsthm}
\usepackage{amssymb}
\usepackage{amscd}
\usepackage{pstricks}
\usepackage{color}
\usepackage{mathpazo}
\usepackage[dvips]{graphicx}
\usepackage{textcomp}

\newtheorem{theorem}{Theorem}
\newtheorem{lemma}{Lemma}
\newtheorem{definition}{Definition}
\newtheorem{corollary}[lemma]{Corollary}
\newtheorem{remark}{Remark}
\newtheorem{fact}{Fact}

\newtheorem{note}{Note}
\newtheorem{conjecture}{Conjecture}

\newcommand{\reals}{{\mathbb{R}}}
\newcommand{\Prob}{{\mathbb{P}}}

\begin{document}
\title{Perturbation approach to scaled type Markov renewal processes with infinite mean.}

\author{Zsolt Pajor-Gyulai\footnote{Budapest University of Technology,Institute of Physics , \textit{pgyzs@math.bme.hu}} \and Domokos Sz\'asz\footnote{Budapest University of Technology, Mathematical Institute ,Budapest, Egry J. u. 1 Hungary H-1111, \textit{szasz@math.bme.hu} }}

\maketitle

\vskip1cm
\begin{abstract}
Scaled type Markov renewal processes generalize classical renewal processes: renewal times come from a one parameter family of probability laws and the sequence of the parameters is the trajectory of an ergodic Markov chain. Our primary interest here is the asymptotic distribution of the Markovian parameter at time $t \to \infty$. The limit, of course, depends on the stationary distribution of the Markov chain. The results, however, are essentially different depending on whether the expectations of the renewals are finite or infinite. If the expectations are uniformly bounded, then we can provide the limit in general (beyond the class of scaled type processes), where the expectations of the probability laws in question appear, too. If the means are infinite, then -- by assuming that the renewal times are rescaled versions of a regularly varying probability law with exponent $0 \le \alpha \le 1$ -- it is the exponent $\alpha$ which emerges in the limits.
\vskip5mm

\end{abstract}

\section{Introduction} Heavy tailed probability distributions have recently arisen in new interesting applications, it is sufficient to mention waiting times in queueing networks like the internet or stock prices. For us the laws with exponents $\alpha =\ 0 \ \textrm{and}\ \frac{1}{2}$ came into play in stochastic  models of physical phenomena as return times to the origin of processes which are proved to behave analogously to random walks on $\mathbb Z^d$ where $ d = 1\ \textit{or}\ 2$ (more concretely,  in stochastic paradigms of two Lorentz disks in the plane or in a quasi-one-dimensional slab, see a forthcoming article of the authors).

Markov renewal models are themselves interesting mathematical objects and, in particular, the model of scaled type renewal processes, suggested by the physical model, seems to be a fortunate notion. As to some history: Markov renewal processes (or Semi Markov processes) were independently introduced in 1954 by L\'evy \cite{L54}, Smith
\cite{S55} and Tak\'acs \cite{T54}.
The basic theory was developed by Pyke \cite{P61a}, \cite{P61b} and then further elaborated among others by Pyke and Schaufele \cite{PS64}, Cinlar \cite{C69}, \cite{C75}, Koroliuk and his coauthors. For a recent treatment see \cite{BL08} or the works of Jannsen \cite{J86},\cite{JL99}. Nevertheless, none of these authors seem to have addressed the situation when the waiting times have infinite means. Indeed, it is not at all clear how to formulate results in general. However, scaled type processes generated by a slowly varying law as suggested by our physical model (cf. Section 6) provide a suitable model for treating these questions.

In this paper we give a comprehensive answer for the most original primary question related to Markov Renewal processes: we determine the asymptotic distribution of the Markovian parameter at time $t \to \infty$. Our main interest is the case when the variables have infinite expectations and the process is of scaled type. To emphasize coherence, we also prove  -- by using our method -- results, already known for the finite mean case. We develop an operator formalism and use some facts from perturbation theory to develop a key lemma from which most of our results follow easily.

In the theory of ordinary renewal processes, the first attempts to extend the well known result of Feller and Smith (\cite{F71},\cite{S54}) to the infinite mean case were performed by Erickson(\cite{E70}), Teugels (\cite{T68}) and Anderson\&Athreya (\cite{AA87}).   As for the Markov renewal process, some partial results has been already obtained, e.g. in \cite{MY01}, many properties of the spent time (age-) process (see Sect 5.) were established under different assumptions for the alternating renewal process. In this paper we show that under certain assumptions, the classical result of Dynkin (\cite{D55}) still holds.

The paper is organized as follows. Section 2 contains our definitions and our key technical result. Sections 3, 4, 5 deal with its consequences while section 6 presents the physics application which drove our attention to the topic. Finally section 7 is devoted to the proofs of our theorems.
\begin{note}
We originally used the name Renewal Process directed by a Markov Chain (RPdMC), but we decided to stick to traditions and to use Markov renewal process. Some authors use the name Semi-Markov process for the whole phenomena, but we only refer to a particular process by this name.
\end{note}

\section{Definitions and basic results}

\subsection{Basic definitions and conditions}
Consider a measurable function $F_{\lambda}(t)\equiv F(\lambda,t):[a,b]\times\reals_+\to[0,1]$ with

\vspace{2mm}
\noindent\textbf{Basic assumptions:}
\begin{equation}\label{F_lambda_distri_func}
\textrm{For fixed $\lambda$, $F_{\lambda}(.)$ is a non-arithmetic distribution function}
\end{equation}
\begin{equation}\label{F_lambda_condition}
\exists\delta>0:\quad\sup_{\lambda\in[a,b]}F_{\lambda}(\delta)<1
\end{equation}
\begin{equation}\label{F_lambda_condition2}
\exists K\in\mathbb{R}_+:\quad\inf_{\lambda\in[a,b]}F_{\lambda}(K)>0
\end{equation}

We will need random variables $X_{\lambda}$ with distribution function $F_{\lambda}$. If $X_{\lambda}$ has expectation, then is denoted by $\mu_{\lambda}$.
\begin{remark}
\noindent Conditions \eqref{F_lambda_condition} and \eqref{F_lambda_condition2} implies that there is no sequence $(\lambda_i)_{i\geq 0}$
that $X_{\lambda_i}$ would converge either to the point mass at zero or to infty in distribution (or - as the limit is non-random - in probability).
\end{remark}
\begin{definition}
The family of distributions defined above is called {\rm scaled-type} if there is a distribution function $F:\reals_+\to[0,1]$ for which
\[
F_{\lambda}(t)=F(\lambda t)
\]
\end{definition}

In this case, the basic assumptions are satisfied if $0<a\leq b<\infty$ and moreover if  $\mu=\int_0^{\infty}xdF(x)$ is finite , then $\mu_{\lambda}=\mu/\lambda$.

Also suppose we have a recurrent Harris chain (cf. e.g. Chapter 5.6 in \cite{D05}) ($\Lambda_0,\Lambda_1,...)$ on $[a,b]$ with transition kernel $g(\lambda_-,A)$. Suppose that this chain has a spectral gap, which means that the spectrum of its transition operator on $L_{\infty}([a,b])$ is bounded away from the unit circle except for the eigenvalue $1$ and  finitely many other eigenvalues on the unit circle. Let $\rho_s$ denote the stationary measure.


\begin{definition}

Suppose $(\lambda_0,\lambda_1, \lambda_2, \dots) \in [a, b]^\mathbb N$. Then the sequence $S_n = \sum_{j=0}^{n-1} X_{\lambda_j}:\ n = 1, 2, \dots$ with $S_0=0$ is called a {\rm Non-Homogeneous Renewal Process (NHRP)} if $X_{\lambda_0},X_{\lambda_1}, X_{\lambda_2}, \dots$ is an independent sequence of random variables such that $\forall j \in \mathbb N$ the distribution of $X_{\lambda_j}$ is $F_{\lambda_j}$. If furthermore $F_{\lambda}(t)=F(\lambda t)$ with some distribution function $F$, then we call the process a {\rm Scaled-type Renewal Process (STRP)}.

\end{definition}

\begin{definition}

The sequence $S_{n} = \sum_{j=0}^{n-1} X_{\Lambda_j}:\ n = 1, 2, \dots$ with the convention $S_{0}=0$ is called a {\rm Markov renewal process} if 
\begin{itemize}
\item $\Lambda_0, \Lambda_1, \Lambda_2, \dots$ is a homogeneous Markov chain introduced above and 
\item for every realization  $\lambda_0, \lambda_1, \dots$ of this Markov chain $S_n:\ n = 0, 1, 2, \dots$ is a non-homogeneous renewal process. 
\end{itemize}
(Notation: if we want to emphasize the dependence of the process on $\lambda_0$, we write $S_{n, \lambda_0}$.)
\end{definition}


Consider a Markov renewal process and -- by complying with the classical renewal terminology -- let $N_{t,\lambda_0}$ denote the number of renewals that occurred before time $t$ (including the one at $t=0$) with initial parameter value $\lambda_0$, i.e.
\begin{equation}\label{felujitasok_szama}
N_{t,\lambda_0}=\inf\{n:S_{n, \lambda_0}> t\}
\end{equation}
and let $U_{\lambda_0}(t)= \mathbb E N_{t,\lambda_0}$. Denote the "type" of the  renewal ongoing at time $t$ by $\Lambda(t) = \Lambda_{N_{t,\lambda_0}-1}$ and the distribution of the parameter $\Lambda(t)$, conditioned on the initial parameter value $\lambda_0$, by $\Phi_{t,\lambda_0}$, i.e.
\begin{equation}\label{def_phi}
\Phi_{t,\lambda_0}(A)=\Prob(\Lambda(t)\in A\in\mathcal{B}([a,b])|\Lambda_0=\lambda_0)
\end{equation}
\begin{note}
$\Lambda(t)$ is a so called Semi-Markov process since it would be a continuous time Markov chain on $[a,b]$, if for every $\lambda$, $F_{\lambda}$ were an exponential distribution function.
\end{note}
By conditioning on the first renewal, the renewal equation writes as
\begin{align}\label{ren_eq_spec}
\Phi_{t,\lambda_0}(A)=&\mathbb{1}_{\{\lambda_0\in A\}}(1-F_{\lambda_0}(t))+\\
\nonumber &+\int_0^t\int_a^b\Phi_{t-s,\lambda_1}(A)g(\lambda_0,d\lambda_1) dF_{\lambda_0}(s)
\end{align}

All the basic phenomena are governed by equations like \eqref{ren_eq_spec}. Since this is not the usual renewal equation, we have to generalize standard renewal theory. Our first result is an existence and uniqueness theorem.
\begin{theorem}\label{ren_eq_unique}
For any measurable function $h_{.,\lambda_0}(A)$ which is bounded on bounded intervals, i.e.
\[
(\forall t>0)(\exists M_t<\infty):\qquad |h_{s,\lambda_0}(A)|<M_t\quad\forall s\in[0,t]
\]
the solution of equation
\begin{align}\label{ren_eq}
\Psi_{t,\lambda_0}(A)=h_{t,\lambda_0}(A)+\int_0^t\int_a^b\Psi_{t-s,\lambda_1}(A)g(\lambda_0,d\lambda_1) dF_{\lambda_0}(s)
\end{align}
exists and is unique among the functions that vanish for $t<0$ and are bounded on bounded intervals.

Moreover, the solution can be given as an infinite series:
\begin{align}\label{solution}
\Psi_{t,\lambda_0}&(A)=h_{t,\lambda_0}(A)+\\
\nonumber&+\sum_{n=1}^{\infty}\int_{[a,b]^n}\int_0^th_{t-s,\lambda_n}(A)d\bigg((\Pi^*)_{i=0}^{n-1}F_{\lambda_i}(s)\bigg)\prod_{i=0}^{n-1}g(\lambda_i,d\lambda_{i+1})
\end{align}
where $\Pi^*$ denotes the convolution product.
\end{theorem}

This form of the solution is troublesome to work with, but we can also write
\begin{equation}\label{solution_nicer}
\Psi_{t,\lambda_0}(A)=\int_{[0,t]\times[a,b]}h_{t-s,\lambda}(A)U_{\lambda_0}(ds,d\lambda)
\end{equation}
where we introduced the functions
\begin{equation}\label{U(.,.)}
U_{\lambda_0}(t,A)=\mathbb{1}_{\{\lambda_0\in A\}}\Theta(t)+g(\lambda_0,A)F_{\lambda_0}(t)+
\end{equation}
\[
+\sum_{n=2}^{\infty}\int_{[a,b]^{n-1}}g(\lambda_{n-1},A)(\Pi^*)_{i=0}^{n-1}F_{\lambda_i}(t)\prod_{i=0}^{n-2}g(\lambda_i,d\lambda_{i+1})
\]
Here $\Theta(t)=1$ if $t>0$ and zero otherwise. This can be further written
\begin{equation}\label{Alt_def_U}
U_{\lambda_0}(t,A)=\mathbb{1}_A(\lambda_0)+\sum_{n=1}^{\infty}\Prob_{\lambda_0}(X_{\Lambda_0}+...+X_{\Lambda_{n-1}}<t;\Lambda_n\in A)
\end{equation}
so $U_{\lambda_0}(t,A)$ is the expected number of jumps into the set $A$ before time $t$ (plus 1 if the process is launched from A).
The integration in \eqref{solution_nicer} is wrt the measure defined by
\begin{align*}
U_{\lambda_0}(A\times [t_1,t_2])=U_{\lambda_0}(t_2,A)&-U_{\lambda_0}(t_1,A)
\end{align*}
for $A\in \mathcal{B}([a,b])$ and $0<t_1<t_2<\infty$.

Note that from \eqref{Alt_def_U} it is clear that $U_{\lambda_0}(t) = U_{\lambda_0}(t,[a,b])$.

\subsection{Laplace transforms}

Introduce the Laplace transform of $F$:
\[
\varphi_{\lambda}(z)=\int_0^{\infty}e^{-zx}dF_{\lambda}(x)=z\int_0^{\infty}e^{-zx}F_{\lambda}(x)dx\qquad z\geq 0
\]

From the last formula, it is easy to see that \eqref{F_lambda_condition} and \eqref{F_lambda_condition2} imply for $z>0$
\begin{equation}\label{Condition_cons}
\sup_{\lambda\in [a,b]}\varphi_{\lambda}(z)<1\qquad\inf_{\lambda\in[a,b]}\varphi_{\lambda}(z)>0
\end{equation}
In the scaled type case, $\varphi_{\lambda}(z)=\varphi(z/\lambda)$, where $\varphi(z)$ is the Laplace transform of the measure $dF(.)$
Also let
\begin{equation}\label{Laplace_U}
\omega_{\lambda_0}(z,A)=\int_{[0,\infty]\times A}e^{-zs}U_{\lambda_0}(ds,d\lambda)=\int_0^{\infty}e^{-zs}U_{\lambda_0}(ds,A)
\end{equation}

For fixed $z$, $\omega_{\lambda_0}(z,.)$ is, of course, a measure on $[a,b]$. By the virtue of \eqref{Alt_def_U} this  can be also written  as
\[
\omega_{\lambda_0}(z,A)=\mathbb{1}_A(\lambda_0)+\sum_{n=1}^{\infty}\mathbb{E}\left(e^{-z(X_{\Lambda_0}+...+X_{\Lambda_{n-1}})}\mathbb{1}_{\{\Lambda_n\in A\}}\Big|\Lambda_0=\lambda_0\right)
\]

Also let $\Xi_{\lambda_0}(z,A)$ be the Laplace transform of $\Psi_{t,\lambda_0}(A)$ in the variable $t$, i.e.
\[
\Xi_{\lambda_0}(z,A)=\int_0^{\infty}e^{-zt}\Psi_{t,\lambda_0}(A)dt
\]
Then by \eqref{solution_nicer}, Fubini's theorem, and the product rule of the Laplace transform,
\begin{equation}\label{Laplace_solution}
\Xi_{\lambda_0}(z,A)=\int_a^b\phi_{\lambda}(z,A)\omega_{\lambda_0}(z,d\lambda)
\end{equation}
where the integration is wrt the measure defined in \eqref{Laplace_U} and  $\phi_{\lambda}(z,A)$ is the Laplace transform of $h_{t,\lambda}(A)$, i.e.
\begin{equation}\label{Laplace_1-F}
\phi_{\lambda}(z,A)=\int_0^{\infty}e^{-zt}h_{t,\lambda}(A)dt
\end{equation}
In all the applications, $h$ is so, that this Laplace transform exists. Clearly, then $\Xi_{\lambda_0}(z,A)$ exists as well.
Despite its simplicity, it is not equation \eqref{Laplace_solution} which proves useful in the sequel. Instead, take the Laplace transform of \eqref{ren_eq} to obtain
\begin{equation}\label{ren_eq_laplace}
\Xi_{\lambda_0}(z,A)=\phi_{\lambda_0}(z,A)+\int_a^b\Xi_{\lambda_1}(z,A)\varphi_{\lambda_0}(z)g(\lambda_0,d\lambda_1)
\end{equation}
$\forall\lambda_0\in[a,b]$ and $\forall A\in\mathcal{B}([a,b])$.
\subsection{Key lemma}
Recall that
\begin{definition}
A positive function $L(t)$ defined on $ \mathbb R_{\ge 0} $ is slowly varying at infinity if
\[
\frac{L(ct)}{L(t)}\to 1\qquad\forall c\in\reals_{>0}
\]
\end{definition}

The key element in the treatment is the following
\begin{lemma}\label{key_lemma}
Whenever $\mu_{\lambda}<\infty$ for all $\lambda\in[a,b]$, or the process is scaled type with a regularly varying ancestor distribution, i.e.
\[
1-F(t)=t^{-\alpha}L(t)\qquad\alpha\in[0,1]
\]
where $L$ is a slowly varying function at infinity, we have
\begin{equation}\label{Xi_asym}
\Xi_{\lambda_0}(z,A)\sim\frac{\int_a^b\phi_{\lambda}(z,A)d\rho_s(\lambda)}{\int_a^b(1-\varphi_{\lambda'}(z))d\rho_s(\lambda')} 
\end{equation}
as $z\to 0$ provided that
\[
\limsup_{z\to 0}\frac{\sup_{\lambda\in[a,b]}\phi_{\lambda}(z,A)}{\int_a^b\phi_{\lambda}(z,A)d\rho_s(\lambda)}<\infty
\]
\end{lemma}
\begin{remark}\label{Lemma1_utan}
The main idea behind Lemma \ref{key_lemma} is that the asymptotic behaviour is independent of the initial state. Thus the asymptotic formulas must be identical with what would be exact if the distribution of $\lambda_0$ was the stationary one.
\end{remark}
\subsection{The operator formalism}
In the proof of Lemma \ref{key_lemma}, we use a perturbation approach in the framework of an operator formalism developed in this section.

As usual, let $L_{\infty}[a,b]$ denote the set of bounded, measurable functions on $[a,b]$. The transition operator of the Markov chain (defined by the kernel $g$ of the previous section), denoted by $P$, operates on this space by
\[
Pf(\lambda)=\int_a^b f(\lambda')g(\lambda,d\lambda')\qquad f\in L_{\infty}([a,b]),\lambda\in[a,b]
\]
Of course, on the adjoint space $\mathcal{M}([a,b])$ (i.e. every signed measure on $[a,b]$ of finite total variation) its effect is given by
\[
\mu P(A)=\int_a^bg(\lambda,A)d\mu(\lambda)\qquad\mu\in\mathcal{M}([a,b]),A\in\mathcal{B}([a,b])
\]

Also define the operator valued functions $\hat{\varphi}(z)$, $\hat{\phi}(z)$, and $\hat{\Xi}(z)$ acting on $L_{\infty}([a,b])$ by
\[
(\hat{\varphi}(z)f)(\lambda)=\varphi_{\lambda}(z)f(\lambda)\qquad (\hat{\phi}(z)\mathbb{1}_A)(\lambda)=\phi_{\lambda}(z,A)
\]
\[
(\hat{\Xi}(z)\mathbb{1}_A)(\lambda)=\Xi_{\lambda}(z,A)
\]
where $f\in L_{\infty}([a,b])$, $\mathbb{1}_A$ is the indicator function of A. In the last two definitions, the operators are defined on the  linear span of step functions in $L_{\infty}([a,b])$.

With these, it can be easily seen that equation \eqref{ren_eq_laplace} is equivalent to the operator equation
\[
\hat{\Xi}(z)=\hat{\phi}(z)+\hat{\varphi}(z) P\hat{\Xi}(z)
\]
This yields the formal solution
\begin{equation}\label{form_sol}
\hat{\Xi}(z)=(I-\hat{\varphi}(z) P)^{-1}\hat{\phi}(z)
\end{equation}
Condition \eqref{Condition_cons} ensures the existence of the inverse for every $z>0$, since $||\hat{\varphi}(z)||=\sup_{\lambda\in[a,b]}\varphi_{\lambda}(z)<1$.

Denote the effect of a measure $\mu$ in $\mathcal{M}([a,b])$ as a functional on an element $f$ of $L_{\infty}([a,b])$ with $(\mu,f)$, i.e.
\[
(\mu,f)=\int_{[a,b]}fd\mu
\]
and note that e.g.
\[
\Xi_{\lambda_0}(z,A)=(\delta_{\lambda_0},\hat{\Xi}(z)\mathbb{1}_A)
\]
where $\delta_{\lambda_0}$ is the point mass concentrated on $\lambda_0$.
In this framework, Lemma 1 can be rephrased as
\begin{lemma}\label{key_lemma2}
Suppose $\mu_{\lambda}<\infty$ $\forall\lambda\in[a,b]$ or that
\[
1-F_{\lambda}(t)=1-F(\lambda t)=(\lambda t)^{-\alpha}L(\lambda t)\qquad\alpha\in[0,1]
\]
where $L$ is a slowly varying function. Then if
\begin{equation}\label{lemma_condition}
\limsup_{z\to 0}\frac{||\hat{\phi}(z)||}{(\rho_s,\hat{\phi}(z)\mathbb{1}_A)}<\infty
\end{equation}
then
\[
(\delta_{\lambda_0},\hat{\Xi}(z)\mathbb{1}_A)\sim \frac{(\rho_s,\hat{\phi}(z)\mathbb{1}_A)}{(\rho_s,(I-\hat{\varphi}(z))\mathbb{1})}\qquad z\to 0
\]
where $\mathbb{1}=\mathbb{1}_{[a,b]}$.
\end{lemma}

\begin{conjecture}
Lemma \ref{key_lemma2} is likely to be true under the somewhat milder condition that $\mu_{\rho_s}=\int_a^b\mu_{\lambda}d\rho_s(\lambda)<\infty$, which allows $\mu_{\lambda}$ to be infinite on a $\rho_s$-negligible set if $\hat{\phi}(z)$ is nice in some sense. The ground of this suggestion is that a $\rho_s$-negligible set cannot have large effect on asymptotic relations. This question does not arise in the scaled type case, so we do not pursuit it in the sequel. However, we mention 
\end{conjecture}
\begin{corollary}\label{key_lemma_corollary}
If the annihilator of $\rho_s$ i.e.
\[
A_{\rho_s}=\{f\in L_{\infty}([a,b]):(\rho_s,f)=0\}
\]
is an invariant subspace of $\hat{\phi}(z)$ for every $z$, then the assertion of Lemma \ref{key_lemma2} holds if  $\mu_{\lambda}=\infty$ only on a $\rho_s$-null set.
\end{corollary}

\section{Generalization of the renewal theorem (asymptotics of $U_{\lambda_0}(t,A)$}
In this section, we investigate the asymptotic behaviour of $U_{\lambda_0}(t,A)$. To do this, note that \eqref{solution_nicer} implies that if $h_{t,\lambda_0}(A)=\mathbb{1}_A(\lambda_0)$, then we have $\Psi_{t,\lambda_0}(A)=U_{\lambda_0}(t,A)$ and also $\phi_{\lambda}(z,A)=\mathbb{1}_A(\lambda)/z$. The assumption \eqref{lemma_condition} is satisfied if $\rho_s(A)>0$. Then Lemma \ref{key_lemma} yields
\begin{equation}\label{Lemma1_to_U}
\omega_{\lambda_0}(z,A)-\mathbb{1}_A(\lambda_0)\sim \frac{1}{\int_a^b(1-\varphi_{\lambda}(z))d\rho_s(\lambda)}\rho_s(A)\qquad z\to 0
\end{equation}
(note the difference between $\omega$ and $\Xi$!). Thus we have
\begin{theorem}\label{Thm_U_asym}
For $A\in\mathcal{B}([a,b])$ with $\rho_s(A)>0$, we have for $\mu_{\rho_s}<\infty$ that
\begin{equation}\label{finite_U_asym}
U_{\lambda_0}(t,A)\sim \frac{t}{\mu_{\rho_s}}\rho_s(A)
\qquad \qquad t \to \infty
\end{equation}
while in the scaled type case for $\alpha\in[0,1)$,
\[
U_{\lambda_0}(t,A)\sim \frac{t^{\alpha}}{L(t)}\frac{\sin(\pi\alpha)/\pi\alpha}{\int_a^b\lambda^{-\alpha}d\rho_s(\lambda)}\rho_s(A)=\frac{\sin(\pi\alpha)/\pi\alpha}{1-F(t)}\frac{\rho_s(A)}{\int_a^b\lambda^{-\alpha}d\rho_s(\lambda)}
\]
Note that if $\alpha=0$, the last factor is one. When $\alpha=1$, one obtains
\[
U_{\lambda_0}(t,A)\sim\frac{t}{\tilde{L}(t)}\frac{\rho_s(A)}{\int_a^b\lambda^{-1}d\rho_s(\lambda)}
\]
where $\tilde{L}=\int_0^t(1-F(s))ds$ varies slowly and  $U_{\lambda_0}(t,A)(1-F(t))\to 0$ in addition.
\end{theorem}

\begin{remark}
If $\rho_s(A)=0$, then $\Prob(\Lambda_n\in A)<C_Ae^{-\gamma n}$, where $\gamma$ is the spectral gap of the Markov chain. Thus by \eqref{Alt_def_U}, we have the estimate
\[
U_{\lambda_0}(t,A)<1+\sum_{n=1}^{\infty}\Prob(\Lambda_n\in A)<\frac{C_A}{1-e^{\gamma}}
\]
where the last inequality implies that only finitely many times does the chain jump to $A$ as $t\to\infty$ with probability one.
\end{remark}

\section{Asymptotic results for $\Phi_{t,\lambda_0}(A)$}
In this special case $h_{t,\lambda}(A)=\mathbb{1}_{\{\lambda\in A\}}(1-F_{\lambda}(t))$, \eqref{solution_nicer} becomes
\begin{equation}\label{solution_spec}
\Phi_{t,\lambda_0}(A)=\int_{[0,t]\times A}(1-F_{\lambda}(t-s))dU_{\lambda_0}(s,\lambda)
\end{equation}
and $\phi_{\lambda}(z,A)=\mathbb{1}_{\{\lambda\in A\}}(1-\varphi_{\lambda}(z))/z$. Thus
\[
\Xi_{\lambda_0}(z,A)\sim\frac{1}{z}\frac{\int_A(1-\varphi_{\lambda}(z))d\rho_s(\lambda)}{\int_a^b(1-\varphi_{\lambda'}(z)))d\rho_s(\lambda')}\qquad z\to 0
\]

Here \eqref{lemma_condition} is satisfied of for every $\lambda_0$
\begin{equation}\label{Phi_felt}
\liminf_{z\to 0}\int_A\frac{1-\varphi_{\lambda}(z)}{1-\inf_{\lambda\in[a,b]}\varphi_{\lambda}(z)}d\rho_s(\lambda)>0
\end{equation}
which holds if $\rho_s(A)>0$. To see this note that in the finite mean case
\eqref{Condition_cons} ensures that $\inf_{\lambda\in[a,b]}\mu_{\lambda}>0$ and \eqref{Phi_felt} flollows from the asymptotic expansion of the $\varphi$'s. In the scaled type case note that the integral in \eqref{Phi_felt} admits the lower bound
\[
\rho_s(A)\frac{1-\varphi(az)}{1-\varphi(bz)}\geq\rho_s(A)\frac{a}{b}>0
\]
due to concavity of $1-\varphi$.
 Our result is
\begin{theorem}\label{fin_exp_asym}
For $A\in\mathcal{B}([a,b])$ with $\rho_s(A)>0$, we have
\begin{equation}\label{finite_exp_markov}
\lim_{t\to\infty}\Phi_{t,\lambda_0}(A)=\frac{1}{\mu_{\rho_s}}\int_A\mu_{\lambda}d\rho_s(\lambda)\qquad\forall\lambda_0\in [a,b]
\end{equation}
if $\mu_{\rho_s}<\infty$. In the scaled type, finite mean case, this becomes
\begin{equation}\label{regular_var_markov}
\lim_{t\to\infty}\Phi_{t,\lambda_0}(A)=\frac{\int_A\frac{1}{\lambda}d\rho_s(\lambda)}{\int_a^b\frac{1}{\lambda'}d\rho_s(\lambda')}\qquad\forall\lambda_0\in [a,b]
\end{equation}
If in the scaled type case $1-F(t)=t^{-\alpha}L(t)$, we have
\[
\lim_{t\to\infty}\Phi_{t,\lambda_0}(A)=\frac{\int_A\lambda^{-\alpha}d\rho_s(\lambda)}{\int_a^b\lambda'^{-\alpha}d\rho_s(\lambda')}
\]
which implies that in the special case $\alpha=0$, the limit is just $\rho_s(A)$.
\end{theorem}

\begin{remark}
\eqref{finite_exp_markov} and \eqref{regular_var_markov} are true for $\rho_s(A)=0$ as well, since $\Phi_{t,\lambda_0}$ is a measure (Apply the result to $A^c$).
\end{remark}

\section{Results for the age process and the residual and total lifetimes}
Let $Y_{t,\lambda_0}$ denote the time since the last renewal occurred and $Z_{t,\lambda_0}$ is the remaining time until the next renewal, i.e.
\[
Y_{t,\lambda_0}=t-S_{N_{t,\lambda_0}}\qquad Z_{t,\lambda_0}=S_{N_{t,\lambda_0}+1}-t
\]
The total lifetime is the sum $C_{t,\lambda_0}=Y_{t,\lambda_0}+Z_{t,\lambda_0}$.

It is easy to see, that $\Prob(Y_{t,\lambda_0}<x)\mathbb{1}_A(\lambda_0)$ satisfies \eqref{ren_eq} with the inhomogeneous term $h_{t,\lambda}^Y(A)=\mathbb{1}_{[0,x]}(t)\mathbb{1}_A(\lambda) (1-F_{\lambda}(t))$. Of course, in the end we will set $A=[a,b]$, but now we need the dependence on $A$ to make $\hat{\phi}$ a linear operator. This yields
\[
\phi_{\lambda}^Y(z,A)=\mathbb{1}_A(\lambda)\int_0^xe^{-zt}(1-F_{\lambda}(t))dt
\]
Since  we can use the bounded convergence theorem for fixed x, we have
\[
\phi_{\lambda}^Y(z,A)\to\mathbb{1}_A(\lambda)\int_0^x(1-F_{\lambda}(t))dt>0\qquad z\to 0
\]
and therefore by Lemma \ref{key_lemma} (since \eqref{lemma_condition} is automatically satisfied),
\[
\Xi_{\lambda_0}^Y(z,[a,b])\sim\frac{\int_a^b\int_0^x(1-F_{\lambda}(t))dtd\rho_s(\lambda)}{\int_a^b(1-\varphi_{\lambda}(z))d\rho_s(\lambda)}
\]

It is also not hard to obtain that $\Prob(Z_{t,\lambda_0}<x)\mathbb{1}_A(\lambda_0)$ also satisfies \eqref{ren_eq} with $h_{t,\lambda}^Z(A)=\mathbb{1}_A(\lambda)(F_{\lambda}(t+x)-F_{\lambda}(t))$, and after some calculation, we get
\[
\lim_{z\to 0}\phi_{\lambda}^Z(z,A)=\lim_{z\to 0}\phi_{\lambda}^Y(z,A)
\]
so $\Xi_{\lambda_0}^Z(z,[a,b])\sim\Xi_{\lambda_0}^Y(z,[a,b])$.

As to $C_{\lambda_0,t}$, one can obtain $h_{t,\lambda}^C(A)=\mathbb{1}_A(\lambda)\mathbb{1}_{[0,x]}(t)(F_{\lambda}(x)-F_{\lambda}(t))$, thus for the Laplace transform $\phi_{\lambda}^C(z,A)\to\mathbb{1}_A(\lambda)\int_0^x(F_{\lambda}(x)-F_{\lambda}(t))dt$ as $z\to 0$ and
\[
\Xi_{\lambda_0}^C(z,[a,b])\sim\frac{\int_a^b\int_0^x(F_{\lambda}(x)-F_{\lambda}(t))d\rho_s(\lambda)}{\int_a^b(1-\varphi_{\lambda}(z))d\rho_s(\lambda)}
\]
\begin{theorem}\label{age_lif_thm}
If $\mu_{\rho_s}<\infty$ then
\[
\left.\begin{array}{c}
\Prob(Y_{t,\lambda_0}<x)\\
\Prob(Z_{t,\lambda_0}<x)
\end{array}
\right\}\to\frac{1}{\mu_{\rho_s}}\int_0^x\int_a^b(1-F_{\lambda}(t'))d\rho_s(\lambda)dt'
\]
and
\[
\Prob(C_{t,\lambda_0}<x)\to\frac{1}{\mu_{\rho_s}}\int_0^x\int_a^b(F_{\lambda}(x)-F_{\lambda}(t'))d\rho_s(\lambda)
\]
\end{theorem}

When the expectations of the waiting times are infinite, there is no proper asymptotic distribution of $Y_{t,\lambda_0}$, all the mass escapes to infinity. Instead, $Y_{t,\lambda_0}/t$ has a limit distribution. The following results are generalizations of the one due to Dynkin about ordinary renewal processes. (Cf. \cite{F71} XIV.3).
\begin{theorem}\label{age_lif_thm2}
If $1-F(t)=t^{-\alpha}L(t)$ with $0<\alpha<1$ in the scaled type case, $Y_{t,\lambda_0}/t$ converges in distribution to the distribution with density function
\[
\frac{\sin(\pi\alpha)}{\pi}x^{-\alpha}(1-x)^{\alpha-1}
\]
while the limit density function of $Z_{t,\lambda_0}/t$ is
\[
\frac{\sin(\pi\alpha)}{\pi}\frac{1}{x^{\alpha}(1+x)}
\]
In the $\alpha=0$ case,
\[
\frac{Y_{t,\lambda_0}}{t}\stackrel{\mathcal{P}}{\to} 1\qquad\frac{Z_{t,\lambda_0}}{t}\stackrel{\mathcal{P}}{\to} 1\qquad t\to\infty
\]
In the $\alpha=1$ case we can only state
\[
\frac{Y_{t,\lambda_0}}{t}\stackrel{\mathcal{P}}{\to} 0\qquad\frac{Z_{t,\lambda_0}}{t}\stackrel{\mathcal{P}}{\to} 0\qquad t\to\infty
\]
\end{theorem}

\begin{remark}
These formulas are identical to the original ones, which means that the presence of different kinds of renewal times is irrelevant asymptotically.
\end{remark}

\section{An application}
Semi-Markov theory is one of the most efficient area of stochastic
processes to generate applications in real-life problems. We cannot give here a complete view of such applications in the fields of (paraphrasing Barbu and Limnios) Economics, Manpower models. Insurance, Finance, Reliability, Simulation, Queuing, Branching processes. Medicine (including survival data). Social Sciences, Language Modelling, Seismic Risk Analysis, Biology, Computer Science, Chromatography and Fluid mechanics, mainly due to the lack of expertise. (see e.g. \cite{J86} or \cite{JL99}))

Therefore, we present the application, which motivated our model the problems treated. Namely  Random Walks with Internal States in one and two dimensions. Shortly we investigated  continuous time random walks with internal states in which the speed parameter was the internal state changing according to a Markov chain at every visit of the random walk to the origin. In two dimensions, it was a paradigm model to the two disk Lorentz process, i.e. two disks wandering in a periodic scatterer configuration and changing energy when they collide with each other.

It can be shown (cf. an upcoming article of the authors) that the return times to the origin are regularly varying with exponent $\alpha=1/2$ in $d=1$, and slowly varying in $d=2$, i.e
\[
1-F^{d=1}(t)\sim C_1t^{-1/2}\qquad 1-F^{d=2}(t)\sim\frac{C_2}{\log t}
\]
The exact values of the constants are not important now. Suppose for ease that the stationary distribution is uniform. In the physical model $[a,b]=[\sqrt{E},\sqrt{2E}]$, where $E$ is the total energy of the two colliding disks.

Our results yield to the expected number of returns to the origin (number of collisions)
\[
U_{\lambda_0}^{d=1}(t)\sim\frac{\sqrt{t}}{\pi C_1 E^{1/4}(2^{1/4}-1)}\qquad U_{\lambda_0}^{d=2}(t)\sim\frac{\log t}{C_2}
\]
Interesting that the energy dependence vanishes in $d=2$ (and this is not because of the special choice of $\rho_s(A)$ which is a good approximation).

The answer to the question concerning is the asymptotic distribution of the speed is simple as well. Note, that due to our assumption of $\rho_s(A)$, the limit distribution has density
\[
\Phi_{\infty,\lambda_0}^{d=1}(\lambda)=\frac{\lambda^{-1/2}}{2E^{1/4}(2^{1/4}-1)}\qquad \Phi_{\infty,\lambda_0}^{d=2}(\lambda)=\frac{1}{\sqrt{E}(\sqrt{2}-1)}
\]
Finally, $Y_{\lambda_0,t}/t$ and $Z_{\lambda_0,t}/t$ has the limit distribution specified in Theorem \ref{age_lif_thm2}. The meaning for $\alpha=0$ is that the current excursion asymptotically dominates the whole process.
\section{Proofs}

\subsection{Used facts}
For the following proofs, we need the so called Abelian-Tauberian theorems (see \cite{F71} XIII.5).
\begin{fact}[Feller]\label{Taub1}
Let H be a measure on $\reals^+$, $\kappa(z)=\int e^{-zx}dH$ the Laplace transform wrt it and $H(x)\equiv H([0,x])$! Then for $\rho\geq 0$,
\[
\frac{\kappa(t/x)}{\kappa(1/x)}\to t^{-\rho}\qquad x\to\infty
\]
and
\[
\frac{H(tx)}{H(x)}\to t^{\rho}\qquad x\to\infty
\]
imply each other, moreover in this case
\begin{equation}\label{taub_1_cons}
\kappa(1/x)\sim H(x)\Gamma(\rho+1)\qquad x\to\infty
\end{equation}
\end{fact}
A popular reformulation of this result is
\begin{fact}\label{Taub2}
If $L$ is slowly varying in infinity and $0\leq\rho<\infty$, then
\[
\kappa(1/x)\sim x^{\rho}L(x)\qquad x\to\infty
\]
and
\[
H(x)\sim\frac{1}{\Gamma(\rho+1)}x^{\rho}L(x)\qquad x\to\infty
\]
imply each other.
\end{fact}
The following result is Example (c) XIII.5 in \cite{F71}
\begin{fact}\label{Taub3}
For $\rho<1$
\[
1-F(x)\sim\frac{1}{\Gamma(1-\rho)}x^{-\rho}L(x)\qquad\textrm{and}\qquad 1-\varphi(z)\sim z^{\rho}L(1/z)
\]
imply each other.
\end{fact}
\subsection{Proof of Existence\&Uniqueness of the solution}
\begin{proof}[Proof of Theorem \ref{ren_eq_unique}]
Suppose we have two such solutions and denote their difference with $\tilde{\Psi}_{t,\lambda_0}(A)$. This function satisfies the homogeneous version of \eqref{ren_eq}:
\[
\tilde{\Psi}_{t,\lambda_0}(A)=\int_0^t\int_a^b\tilde{\Psi}_{t-s,\lambda_1}(A)g(\lambda_0,d\lambda_1)dF_{\lambda_0}(s)
\]
and clearly $|\tilde{\Psi}_{t,\lambda_0}(A)|<2M_t$ if $|\Psi_{s,\lambda_0}(A)|<M_t$ for $s<t$. If we iterate $n$ times, then through a little manipulation (can be checked by induction), we get
\[
\tilde{\Psi}_{t,\lambda_0}(A)=\int_{[a,b]^n}\int_0^t\tilde{\Psi}_{t-s,\lambda_n}(A)d\bigg((\Pi^*)_{i=0}^{n-1}F_{\lambda_i}(s)\bigg)\prod_{i=0}^{n-1}g(\lambda_i,d\lambda_{i+1})
\]
where $\Pi^*$ denotes the convolution product, so
\[
|\tilde{\Psi}_{t,\lambda_0}(A)|<2M_t\int_{[a,b]^n}\int_0^td\bigg((\Pi^*)_{i=0}^{n-1}F_{\lambda_i}(s)\bigg)\prod_{i=0}^{n-1}g(\lambda_i,d\lambda_{i+1})=
\]
\[
=2M_t\int_{[a,b]^n}(\Pi^*)_{i=0}^{n-1}F_{\lambda_i}(t)\prod_{i=0}^{n-1}g(\lambda_i,d\lambda_{i+1})=2M_t\Prob(S_{n,\lambda_0}<t)
\]
which goes to zero as $n\to\infty$ for all $t$ by \eqref{Alt_def_U} if $U_{\lambda_0}(t)<\infty$ for every finite $t$.  But this follows from the fact it is clearly less than the renewal function of a classical renewal process with distribution function
\[
F(x)=\sup_{\lambda\in[a,b]}F_{\lambda}(x)
\]
which is not the point mass at zero by condition \eqref{F_lambda_condition}. Now the statement follows by the result of ordinary renewal theory.

From the proof of uniqueness, one can deduce that if we iterate in the inhomogeneous equation \eqref{ren_eq}, then the remainder term converges to zero. Thus after some calculation, we get exactly the solution given in the theorem. The convergence of the series \eqref{solution} can be checked by noticing
\[
\Psi_{t,\lambda_0}(A)\leq \tilde{M}_tU_{\lambda_0}(t)
\]
where $|h_{s,\lambda_0}(A)|<\tilde{M}_t$ for $s<t$.
\end{proof}



\subsection{Proof of the Key Lemma}

\begin{proof}[Proof of Lemma \ref{key_lemma2}]
Note that
\begin{equation}\label{Perturbation}
I-\hat{\varphi}(z) P=I-P+(I-\hat{\varphi}(z))P
\end{equation}
We will treat the second term as an asymptotic perturbation, where the parameter of the perturbation is $z$.

First consider the $\mu_{\lambda}<\infty$ case. Then for $f\in L_{\infty}([a,b])$
\[
\{(I-\hat{\varphi}(z))f\}(\lambda)=(1-\varphi_{\lambda}(z))f(\lambda)=z\mu_{\lambda}f(\lambda)+o_{\lambda}(z)
\]
where $o_{.}(z)$ is a vector for which $||o_{.}(z)||/z\to 0$ as $z\to 0$. Thus
\begin{equation*}
(I-\hat{\varphi}(z))P=z\mu P+o(z)
\end{equation*}
where $\mu$ is the operator on $L_{\infty}([a,b])$ defined by $(\mu f)(\lambda)=\mu_{\lambda}f(\lambda)$ and the meaning of $o(z)$ is straightforward. Since $1$ is a simple isolated eigenvalue of $P$, which is stable under the perturbation due to the assumed spectral gap (and to the number of eigenvalues on the unit circle being finite), using Theorem 2.6 in Chapter VIII in \cite{K66}, we have that
\begin{equation}\label{spectral_dec}
I-\hat{\varphi}(z)P=(cz+o(z))(\Pi+o(1))+K(z)
\end{equation}
where $\Pi f=(\rho_s,f)\mathbb{1}$, $\nu\Pi=(\nu,\mathbb{1})\rho_s$, and $K(z)$ is the operator arising from the rest of the spectra and projects to the annihilator $A_{\rho_s}$ of $\rho_s$ (see Conjecture 1). Its essential property is that the part of the spectra it is representing is bounded away from zero as $z\to 0$. $o(1)$ is here an operator converging to zero in norm as $z\to 0$.

By $(I-P)\mathbb{1}=0$, one obtains from \eqref{Perturbation} and \eqref{spectral_dec}
\begin{equation}\label{compute_c}
(\rho_s,(I-\hat{\varphi}(z))P\mathbb{1})=(cz+o(z))(\rho_s,\mathbb{1})^2+o(z)=cz+o(z)
\end{equation}
since $(\rho_s,K(z)\mathbb{1})=o(z)$. To see this, note that
\[
0=(\rho_s+z\rho_1+o(z),K(z)(\mathbb{1}+z\mathbb{1}_1+o(z)))
\]
where $\mathbb{1}(z)=\mathbb{1}+z\mathbb{1}_1+o(z)$ is the perturbed right eigenvector that corresponds to the unperturbed eigenvalue $1$. (These asymptotics are guaranteed by the theorem cited above.) After rearrangement,
\begin{equation}\label{szendvics}
(\rho_s,K(z)\mathbb{1})=-z((\rho_1,K(z)\mathbb{1})+(\rho_s,K(z)\mathbb{1}_1))+o(z)=o(z)
\end{equation}
since $||K(z)\mathbb{1}||, ||\rho_sK(z)||\to 0$.

Using the formula \eqref{spectral_dec},
\[
(I-\hat{\varphi}(z)P)^{-1}=\frac{1}{cz+o(z)}(\Pi+o(1))+\mathcal{O}(1)
\]
where $\mathcal{O}(1)$ is a bounded operator which comes from the spectra of $K(z)$ being bounded away from zero. With a little arrangement and application of \eqref{compute_c},
\[
(I-\hat{\varphi}(z)P)^{-1}=\frac{\Pi}{(\rho_s,(I-\hat{\varphi}(z))\mathbb{1})}(1+o(1))+\mathcal{O}(1)
\]
where we used $P\mathbb{1}=\mathbb{1}$ and $o(1)$ is just a real valued function converging to zero as $z\to 0$.

This yields by \eqref{form_sol}
\[
\hat{\Xi}(z)=(I-\hat{\varphi}(z)P)^{-1}\hat{\phi}(z)=\frac{\Pi\hat{\phi}(z)}{(\rho_s,(I-\hat{\varphi}(z))\mathbb{1})}(1+o(1))+\mathcal{O}(1)\hat{\phi}(z)
\]
and finally by $\rho_s\Pi=\rho_s$
\begin{align*}
(\delta_{\lambda_0},\hat{\Xi}(z)\mathbb{1}_A)&\left(\frac{(\rho_s,\hat{\phi}(z)\mathbb{1}_ A)}{(\rho_s,(I-\hat{\varphi}(z))\mathbb{1})}\right)^{-1}=\\
&=1+o(1)+o(1)\frac{(\delta_{\lambda_0},\mathcal{O}(1)\hat{\phi}(z)\mathbb{1}_ A)}{(\rho_s,\hat{\phi}(z)\mathbb{1}_A)}
\end{align*}
By the assumption of the theorem, the explicitly written factor is bounded for small $z$, so the whole expression goes to 1 as $z\to 0$.

In the infinite mean but regularly or slowly varying scaled type case for $\alpha\in[0,1)$,
\begin{align*}
1-\varphi_{\lambda}(z)&=\left(\frac{z}{\lambda}\right)^{\alpha}\Gamma(1-\alpha)L\left(\frac{\lambda}{z}\right)(1+o(1))=\\
&=z^{\alpha}L(1/z)\Gamma(1-\alpha)\lambda^{-\alpha}(1+o(1))
\end{align*}
where the first equation is due to Fact \ref{Taub3}. Thus
\[
(I-\hat{\varphi}(z))P=z^{\alpha}L(1/z)M_{\alpha}P+o(z^{\alpha}L(1/z))
\]
where $(M_{\alpha}f)(\lambda)=\Gamma(1-\alpha)\lambda^{-\alpha}f(\lambda)$. Repeating the finite mean proof with $z\leftrightarrow z^{\alpha}L(1/z)$ gives the desired result.

In the remaining $\alpha=1$ case, we have by the Lemma on p.280 in \cite{F71}, that $H(t)=\int_0^t(1-F(s))ds$ is a slowly varying function. Thus by Fact \ref{Taub2},
\[
1-\varphi(z)=zH(1/z)(1+o(z))
\]
so $(1-\hat{\varphi}(z))P=zH(1/z)M_1+o(zH(1/z))$. Note that $zH(1/z)\to 0$ as $z\to 0$, and again by the virtue of the finite mean proof, we are ready.
\qed
\end{proof}

\begin{proof}[Proof of Corollary \ref{key_lemma_corollary}]
Introduce $\mu_{\infty}=\{\lambda\in[a,b]: \mu_{\lambda}=\infty\}$. What we will show is that $\rho_s(\mu_{\infty})=0$ implies that $\mu_{\infty}$ can be almost literally  dropped  from the state space and thus the assertion. 

If $A_{\rho_s}$ is an invariant subspace of $\hat{\phi}(z)$, then it is also an invariant subspace of $\hat{\Xi}(z)$ as well and thus by $\mathbb{1}_{A\cap\mu_{\infty}}\in A_{\rho_s}$
\[
(\rho_s,\hat{\Xi}(z)\mathbb{1}_A)=(\rho_s,\hat{\Xi}(z)\mathbb{1}_{A\cap \mu_{\infty}^c})
\]

Note also that we can assume $\lambda_0\notin \mu_{\infty}$ since with probability one, ther is an $n$ for which $\Lambda_n\notin \mu_{\infty}$ and we can consider the process launched from there. Then

\[
(\lambda_0,\hat{\Xi}(z)\mathbb{1}_A)=(\lambda_0,\hat{\Xi}(z)\mathbb{1}_{A\cap\mu_{\infty}^c})
\]
This implies that  we only have to work in the subspace spanned by the functions in $L_{\infty}([a,b])$ that does vanish on $\mu_{\lambda}$.
\qed
\end{proof}
\subsection{Proof of the results for $U_{\lambda_0}(t,A)$}
\begin{proof}[Proof of Theorem \ref{Thm_U_asym}]
In the finite mean case we $\rho_s$-almost everywhere have
\[
1-\varphi_{\lambda}(z)=\mu_{\lambda}z+o_{\lambda}(z)
\]
where $\int_a^bo_{\lambda}(z)d\rho_s(\lambda)=o(z)$. This latter can be seen by noting that $\int_a^b\varphi_{\lambda}d\rho_s(\lambda)$ is the Laplace transform of the mixture of $F$-s with respect to $\rho_s(\lambda)$ which is a proper distribution function.

Plugging this to \eqref{Lemma1_to_U} and observing that Corollary \ref{key_lemma_corollary} applies here, we obtain
\[
\omega_{\lambda_0}(z,A)-\mathbb{1}_A(\lambda_0)\sim\frac{1}{z}\frac{\rho_s(A)}{\int_a^b\mu_{\lambda}d\rho_s(\lambda)+o(1)}\sim\frac{1}{z}\frac{\rho_s(A)}{\mu_{\rho_s}}
\]
Using Fact \ref{Taub1}, the proof is ready.

To see the the case when $\alpha\in(0,1)$, note that by Fact \ref{Taub3} and the bounded convergence theorem,
\[
\int_a^b(1-\varphi_{\lambda}(z))d\rho_s(\lambda)=\Gamma(1-\alpha)z^{\alpha}L(1/z)(1+o(1))\int_a^b\lambda^{-\alpha}d\rho_s(\lambda)
\]
By virtue of Fact \ref{Taub2} and by noting that
\[
\frac{1}{\Gamma(1-\alpha)\Gamma(1+\alpha)}=\frac{\sin(\pi\alpha)}{\pi\alpha}
\]
the statement of the theorem is obtained.
In the $\alpha=0$ case,
\[
1-\varphi(az)\leq\int_a^b(1-\varphi_{\lambda}(z))d\rho_s(\lambda)\leq 1-\varphi(bz)
\]
where both the lower and upper bounds are $\sim 1-\varphi(z)$ since they are slowly varying by Fact \ref{Taub3}.
For the remaining $\alpha=1$ case, we again have by \cite{F71} p.280 that
\[
\int_0^t(1-F(s))ds=\tilde{L}(t)
\]
is a slowly varying function. Then by Fact \ref{Taub3},
\[
1-\varphi(z)\sim z\tilde{L}(1/z)\qquad z\to 0
\]
so
\[
\omega_{\lambda_0}(z,A)-\mathbb{1}_A(\lambda_0)\sim\frac{1}{z\tilde{L}(1/z)}\frac{\rho_s(A)}{\int_a^b\lambda^{-1}d\rho_s(\lambda)}
\]
This finishes the proof again by Fact \ref{Taub1}. 

To check the last assertion, note that
\[
U_{\lambda_0}(t,A)(1-F(t))=U_{\lambda_0}(t,A)t^{-1}\tilde{L}(t)\frac{t(1-F(t))}{\int_0^t(1-F(s))ds}
\]
Here the first term is finite by what just has been proved while
\[
\int_0^t(1-F(s))ds=t(1-F(t))+\int_0^tsdF(s)
\]
\qed
\end{proof}

\subsection{Proof of the results for $\Phi_{t,\lambda_0}(A)$}
\begin{proof}[Proof of Theorem \ref{fin_exp_asym}]
Since $\mu_{\lambda}<\infty$ except for a $\rho_s$ negligible set, we again have the asymptotic expansion
\[
\varphi_{\lambda}(z)=1-\mu_{\lambda}z+o_{\lambda}\left(z\right)\qquad z\to 0
\]
so by noting that again Corollary $\ref{key_lemma_corollary}$ applies
\[
\Xi_{\lambda_0}(z,A)\sim\frac{1}{z}\frac{\int_A\mu_{\lambda}d\rho_s(\lambda)}{\int_a^b\mu_{\lambda'}d\rho_s(\lambda')}\equiv \frac{1}{z}K(A)\qquad z\to 0
\]
By \eqref{taub_1_cons}, this implies
\begin{equation}\label{utso}
\lim_{t\to\infty}\frac{1}{t}\int_0^t\Phi_{t,\lambda_0}(s)ds=K(A)>0
\end{equation}
which equivalent to what we are seeking by simple arguments.


In the scaled type case we get the desired formula by  $\mu_{\lambda}=\mu/\lambda$.

In the scaled type, regularly or slowly varying case, the proof is essentially similar to the proof of Theorem \ref{Thm_U_asym}.
\qed
\end{proof}



\subsection{Proof of age and lifetime results}
\begin{proof}[Proof of Theorem \ref{age_lif_thm}]
Through the same procedure as in the proof of Theorem \ref{fin_exp_asym} in the finite mean case.
\qed
\end{proof}

\begin{proof}[Proof of Theorem \ref{age_lif_thm2}]
Using Theorem \ref{Thm_U_asym}, we have that if $\alpha\in[0,1)$,
\begin{equation}\label{U_help}
(1-F(t))U_{\lambda_0}(t)\to\frac{\sin\pi\alpha}{\pi\alpha}\frac{1}{\int_a^b\lambda^{-\alpha}d\rho_s(\lambda)}
\end{equation}
By the same arguments as in \cite{F71} p.472, we have
\begin{align*}
\Prob(tx_1<&Y_{\lambda_0,t}<tx_2)=\\
&=\sum_{n=0}^{\infty}\Prob(\cup_{y\in[1-x_2,1-x_1]}\{S_{\lambda_0,n}=ty\}\cap\{X_{\Lambda_{n+1}}>t(1-y)\})
\end{align*}
which can be seen to equal
\[
\int_{[1-x_2,1-x_1]\times[a,b]}(1-F(\lambda t(1-y)))U_{\lambda_0}(tdy,d\lambda)
\]
By \eqref{U_help}, this is asymptotically equal to
\begin{align}\label{asympt_dynkin}
\frac{\sin(\pi\alpha)}{\pi\alpha}&\frac{1}{\int_a^b\lambda^{-\alpha}d\rho_s(\lambda)}\cdot\\
\nonumber&\cdot\int_{[1-x_2,1-x_1]\times[a,b]}\frac{(1-F(\lambda t(1-y)))}{1-F(t)}\frac{U_{\lambda_0}(tdy,d\lambda)}{U_{\lambda_0}(t)}
\end{align}
In the $\alpha\in(0,1)$ case, the first term in the integral approaches $\lambda^{-\alpha}(1-y)^{-\alpha}$, while
\begin{equation}\label{approach}
\frac{{U}_{\lambda_0}(ty,A)}{U_{\lambda_0}(t)}\to y^{\alpha}\rho_s(A)\qquad\Rightarrow\qquad \frac{U_{\lambda_0}(tdy,d\lambda)}{U_{\lambda_0}(t)}\to\alpha y^{\alpha-1}d\rho_s(\lambda)dy
\end{equation}
as $t\to\infty$. The latter can be seen by noting that Fact \ref{Taub1} and Theorem \ref{Thm_U_asym} yield
\[
\frac{1}{U_{\lambda_0}(t)}\int_{[0,{\infty}]\times A}e^{-zy}U_{\lambda_0}(tdy,d\lambda)=\frac{\omega_{\lambda_0}(z/t,A)}{U_{\lambda_0}(t)}\to\frac{\Gamma(\alpha+1)}{z^{\alpha}}\rho_s(A)
\]
as $t\to\infty$, which is the Laplace transform in the time variable of the measure in \eqref{approach}.

If $\alpha=0$, then the first term goes to one everywhere except $y=1$, while the Laplace transform above is just $1$, which means that the underlying measure converges weakly to the point mass at $y=0$.

Because of monotonicity, the approach is uniform and thus we have for $\alpha\in(0,1)$, that
\[
\lim_{t\to\infty}\Prob(tx_1<Y_{\lambda_0,t}<tx_2)=\frac{\sin(\pi\alpha)}{\pi}\int_{1-x_2}^{1-x_1}(1-y)^{-\alpha}y^{\alpha-1}dy
\]


If $\alpha=0$ and we choose $x_1=0$, then we get zero. Since $x_2>0$ is arbitrary, the desired result is obtained.

In the remaining $\alpha=1$ case, Theorem \ref{Thm_U_asym} implies
\[
t^{-1}U_{\lambda_0}(t)\int_0^t(1-F(s))ds\to\frac{1}{\int_a^b\lambda^{-1}d\rho_s(\lambda)}
\]
so instead of \eqref{asympt_dynkin}, we have
\[
\frac{1}{\int_a^b\lambda^{-1}d\rho_s(\lambda)}\int_{[1-x_2,1-x_1]\times[a,b]}\frac{1-F(\lambda t(1-y))}{t^{-1}\int_0^t(1-F(s))ds}\frac{U_{\lambda_0}(tdy,d\lambda)}{U_{\lambda_0}(t)}
\]
Similarly as before, the measure wrt we are integrating can be shown to weakly converge  to the point mass on $y=1$. If $y\neq 1$, the first fraction in the integrand is asymptotically equal to
\[
\frac{1}{\lambda(1-y)}\frac{1-F(t)}{t^{-1}\int_0^t(1-F(s))ds}
\]
which can be shown to approach zero as $t\to\infty$ by partial integration. If we set $x_2=1$, then since $x_1>0$ is arbitrary, the proof is ready.

The result for the residual lifetime can be obtained through similar modification of the above calculation as in \cite{F71}.
\qed
\end{proof}

\noindent\textbf{Acknowledgement}

We would like to thank P. N\'andori for pointing out numerous typos and helping us improve the original manuscript. D. Sz. is also grateful to Hungarian National Foundation for Scientific Research grants No. T 046187, K 71693,
NK 63066 and TS 049835


\begin{thebibliography}{}

\bibitem{AA87} K.K. Anderson, K.B. Athreya \textit{A renewal theorem in the infinite mean case.} Ann. Prob. \textbf{15}, 388-393 (1987)


\bibitem{BL08} V.S. Barbu, N. Limnios \textit{Semi-Markov Chains and Hidden Semi-Markov Models toward Applications, Their Use in Reliability and DNA Analysis}, Springer (ISBN  978-0-387-73171-1), (2008)





\bibitem{C69} E. Cinlar. \textit{Markov renewal theory.} Adv. in Appl. Probab., 1:123-187, (1969)

\bibitem{C75} E. Cinlar. \textit{Introduction to Stochastic Processes.} Prentice Hall, New York, (1975)

\bibitem {D05} R. Durrett. \textit{Probability. Theory\&Examples}, Duxburry Press, 3rd Edition, (2005)



\bibitem{D55} E.B. Dynkin \textit{Some limit theorems for sums of independent random variables with infinite mathematical expectations.} Izv.Akad.Nauk. SSSR Ser.Mat. \textbf{19}, 247-266 (1955) (in Russian). English translation: Selected Translations, Math. Statist. Prob.\textbf{1},171-189 (1961)

\bibitem{E70} K.B.Erickson \textit{Strong renewal theorems with infinite mean} Trans. Amer. Math. Soc. \textbf{151}, 263-291 (1970)

\bibitem{F71} W. Feller. \textit{An Introduction to Probability Theory and Its Applications, Volume II}, 2nd Ed. John Wiley\&Sons, Inc. (1971)



\bibitem{J86} Janssen J. (ed) \textit{Semi-Markov Models. Theory and Applications} Plenum Press New York. (1986)

\bibitem{JL99} Janssen J., Limnios N. \textit{Semi-Markov Models and Applications.} Kluwer Academic New York. (1999)

\bibitem{K66}T. Kato \textit{Perturbation Theory for linear operators}, Die Grundlehren der mathematischen Wissenschaften in Einzeldarstellungen Band 132, Springer (1966)






\bibitem{L54} P. L\'evy. \textit{Processus semi-markoviens.} In Proc. of International Congress of Mathematics, Amsterdam (1954)

\bibitem{MY01} K.V.Mitov, N.M.Yanev \textit{Limit theorems for alternating renewal processes in the infinite mean case} Adv. Appl. Prob. \textbf{33}, 896-911 (2001)






\bibitem{P61a} R. Pyke. \textit{Markov renewal processes: definitions and preliminary properties.} Ann. Math. Statist., 32:1231-1241, (1961)

\bibitem{P61b} R. Pyke. \textit{Markov renewal processes with finitely many states.} Ann. Math. Statist., 32:1243-1259, (1961).

\bibitem{PS64} R. Pyke and R. Schaufele. \textit{Limit theorems for Markov renewal processes.} Ann. Math. Statist., 35:1746-1764, (1964)




\bibitem{S54} W.L. Smith \textit{Asymptotic renewal theorems} Proc. Roy. Soc. Edinburgh Sect. A \textbf{64} 9-48. (1954)

\bibitem{S55} W.L. Smith. \textit{Regenerative stochastic processes.} Proc. R. Soc. Lond. Ser. A Math. Phys. Eng., 232:6-31, (1955)






\bibitem{T54} L. Takacs. \textit{Some investigations concerning recurrent stochastic processes of a certain type} Magyar Tud. Akad. Mat. Kutato Int.  Kzl., 3:115-128, (1954)

\bibitem{T68} J.L. Teugels \textit{Renewal theorems when the first or the second moment is infinite.} Ann. Math. Statist. \textbf{39},1210-1219 (1968)
\end{thebibliography}
\end{document}